\documentclass[12pt]{amsart}
\usepackage{amssymb}
\usepackage[top=1in, bottom=1in, left=1in,right=1in]{geometry}

\usepackage[utf8]{inputenc}
\usepackage[T1]{fontenc}
\usepackage{mathpazo,courier}
\usepackage[scaled]{helvet}
\usepackage{mathdots}
\usepackage{bm}

\usepackage[usenames,dvipsnames]{xcolor}
\definecolor{darkblue}{rgb}{0.0, 0.0, 0.55}
\usepackage{hyperref}

\usepackage{enumerate}

\newcommand\N{\mathbb N}

\newcommand\R{\mathbb R}
\newcommand\C{\mathbb C}

 \renewcommand\t{\mathbf{t}}
 \newcommand\x{\mathbf{x}}
 \newcommand\y{\mathbf{y}}
 \newcommand\blambda{{\bm \lambda}}

\usepackage{graphicx}

\newcommand\al\alpha
\newcommand\be\beta
\newcommand\de\delta
\newcommand\om\omega
\newcommand\ep\varepsilon
\newcommand\la\lambda
\newcommand\La\Lambda
\newcommand\De\Delta
\newcommand\Ph\Phi
\newcommand\Ps\Psi
\newcommand\ph\varphi
\newcommand\ps\psi
\newcommand\ze\zeta
\newcommand\rh\varrho
\newcommand\Om\Omega

\DeclareMathOperator\ev{ev}

\theoremstyle{definition}
\newtheorem{prop}{Proposition}
\newtheorem{thm}[prop]{Theorem}
\newtheorem{cor}[prop]{Corollary}
\newtheorem{lem}[prop]{Lemma}
\newtheorem{df}[prop]{Definition}

\newtheorem{rem}[prop]{Remark}

\newtheorem{remark}[prop]{Remark}
\newtheorem{ex}[prop]{Example}
\newtheorem{problem}[prop]{Problem}
\newtheorem*{problem*}{Problem}

\title{Generalized eigenvalue methods for Gaussian quadrature rules}

\author[G. Blekherman]{Grigoriy Blekherman}
\address{School of Mathematics, Georgia Institute of Technology, 686 Cherry Street,
Atlanta, GA 30332-0160, USA}
\email{greg@math.gatech.edu}

\author[M. Kummer]{Mario Kummer}
\address{Technische Universit\"at Berlin, Institut f\"ur Mathematik, Sekretariat MA 3-2, Stra\ss{}e des 17. Juni 136, 10623 Berlin, Germany}
\email{kummer@tu-berlin.de}

\author[C. Riener]{Cordian Riener}
\address{Department of Mathematics and Statistics, Faculty of Science and Technology, Uit -- The Arctic University of Norway, 9037 Tromsø, Norway}
\email{cordian.riener@uit.no}

\author[M. Schweighofer]{Markus Schweighofer}
\address{Fachbereich Mathematik und Statistik, Universität Konstanz, 78457 Konstanz, Germany}
\email{markus.schweighofer@uni-konstanz.de}

\author[C. Vinzant]{Cynthia Vinzant}
\address{Department of Mathematics, North Carolina State University, Box 8205, NC State University 
Raleigh, NC 27695-8205, USA}
\email{clvinzan@ncsu.edu}

\date{\today}

\begin{document}
\begin{abstract}
A quadrature rule of a measure $\mu$ on the real line represents 
a convex combination of finitely many evaluations at points, called nodes, 
that agrees with integration against $\mu$ 
for all polynomials up to some fixed degree.  In this paper, we present a bivariate polynomial 
whose roots parametrize the nodes of minimal quadrature rules for 
measures on the real line. We give two symmetric determinantal formulas for 
this polynomial, which translate the problem of finding the nodes to 
solving a generalized eigenvalue problem.  
\end{abstract}
\maketitle


\section{Introduction}\label{intro}

Given a (positive Borel) measure $\mu$ on $\R$, a classical problem in numerical analysis is to 
approximate the integral with respect to $\mu$ of a suitably well-behaved function $f$. One approach 
is via so called \textit{quadrature rules}. 
These approximate the integral by a weighted sum of function values at specified points.  One classical construction for quadrature rules designed to approximate the integral of continuous functions 
consists of demanding  an exact evaluation of the integral for all polynomials of degree $\leq D$.
If the moments of $\mu$ exist and are finite, then this amounts to finding a measure supported on finitely-many points whose moments agree with those of $\mu$ up to degree $D$. 

We use $\t$ as a formal variable on the real line and 
and write $\R[\t]_{\leq D}$ for the vector space of real polynomials of degree at most $D$.
For $k \in \N_0$, we denote the $k$th moment of $\mu$, if it exists and is finite, by  
\[m_k  \ \ = \ \  \int \t^k \  d\mu.  \]

\begin{df}\label{dfqr}
Suppose $D$ is a positive integer and $\mu$ is a measure on $\R$ whose moments up to degree $D$ exist and are finite.
For $x \in \R \cup \{\infty\}$, define the linear function \[\ev_x:\R[\t]_{\leq D} \rightarrow \R\] as follows. 
For $f = \sum_{k=0}^Df_k \t^k \in \R[\t]_{\leq D}$ with $f_0,\ldots,f_D\in\R,$
\[
\ev_x(f) = f(x) \    \text{ for $x\in \R$ \ \ \ \ and }  \ \ \ \  \ev_{\infty}(f) = f_D.
\]
We sometimes write $\ev_{\infty}^D$ for $\ev_{\infty}$ to emphasize its dependence on $D$. 
A \textbf{quadrature rule} of degree $D$ for $\mu$ is a finite set $N\subset \R\cup\{\infty\}$ together with a 
function $w\colon N\to\R_{>0}$ with 
\[\int f \ d \mu \  = \ \sum_{x\in N}w(x) \ev_x(f) \ \ \ \   \text{ for all }\ \ \ \ f\in\R[\t]_{\leq D}.  \]
We call $N$ the \textbf{nodes} of the quadrature rule.
\end{df}
\begin{rem}
In Definition \ref{dfqr} we allow nodes at infinity. This is not desirable for application in numerical analysis, since one cannot generally evaluate functions at these nodes. However, compactifying the real line  by adding $\infty$ makes certain arguments easier. Although our statements are phrased with this more general notion of quadrature, our main result presented below actually provides a tool to explicitly distinguish the cases of quadrature rules with all nodes real
from the case where one of the nodes is $\infty$. Furthermore, it is a classical theorem that quadrature rules for $\mu$ with no nodes at infinity exist for every degree $D$ provided the moments of $\mu$ exist and are finite up to the same degree
(see e.g. \cite[Theorem 5.8]{la1}, \cite[Theorem 5.9]{la2}, \cite[Theorem 1.24]{sch}).
\end{rem}
We say the measure $\mu$ is \textbf{non-degenerate} in degree $d$ if its moments $m_k$ are 
finite up to degree $k=2d$ and for every nonzero, nonnegative polynomial 
$ f\in \R[\t]_{\leq 2d}$ we have $\int f \ d\mu >0$.  Since in one variable any nonnegative polynomial $f$ is sum of squares of polynomials,
 this is equivalent to demanding that $\int p^2 \ d\mu >0$, for every $0\neq p\in \R[\t]_{\leq d}$. This property can be checked quite conveniently in the following way.
 
 \begin{df}  \label{def:Mmatrix}
Consider the quadratic form $p\in \R[\t]\mapsto\int p^2 d\mu$ and restrict it to  $\R[\t]_{\leq d}$.
With respect to the monomial basis $1,\t,\hdots, \t^d$ for $\R[\t]_{\leq d}$, 
this quadratic form  is represented by the $(d+1)\times (d+1)$ 
Hankel matrix with $(i,j)$th entry $m_{i+j-2}$:
\[M_d \ \ = \ \
\begin{pmatrix}
m_0 		& m_1 	& m_2 	& \hdots 		&m_d 	\\
m_1 		& m_2 	& \iddots	&			&m_{d+1} \\
m_2 		& \iddots	& 		&\iddots		&\vdots	\\
\vdots	&		&\iddots	&			&m_{2d-1}\\
m_d 		&m_{d+1}	&\hdots	&m_{2d-1}	&m_{2d}
\end{pmatrix}.\]
With this notation, a measure possessing finite moments up to degree $2d$ is non-degenerate in degree $d$ if and only if 
$\det(M_d)\neq 0$.
\end{df}

From the point of view of numerical analysis it is desirable to have a quadrature rule 
that is exact up to a certain degree and requires the fewest number of evaluations possible.
Such a quadrature rule is called a \textit{Gaussian quadrature rule} for $\mu$. 

It is known that when $D=2d+1$ is odd and $\mu$ is non-degenerate in degree $d$, there is a unique quadrature rule for $\mu$ 
with $d+1$ nodes, and none with fewer nodes \cite[Cor.~9.8]{sch}.   
The nodes can be found as follows.  Let $M'_d$ denote the $(d+1)\times(d+1)$ matrix 
representing the quadratic form $p\mapsto \int \t \cdot p^2 d \mu$ with respect 
to the monomial basis $1,\t,\hdots, \t^d$ of $\R[\t]_{\leq d}$, meaning that $(M'_d)_{i,j}$ equals $m_{i+j-1}$. Then the $d+1$ nodes of the unique Gaussian quadrature rule 
for $\mu$ in degree $2d+1$ are the $d+1$ roots of the univariate polynomial $\det(\x M_d - M_d') $, see \textit{e.g.}~\cite[form. 2.2.9, p. 27]{sze}.
Since $M_d$ and $M'_d$ are real symmetric matrices and $M_d$ is positive definite, 
this writes the problem of finding these $d+1$ nodes as a 
\emph{generalized eigenvalue problem}.

In this paper we focus on the case when $D=2d$ is even. In this case, there is a one-parameter family of quadrature rules for $\mu$  with $d+1$ nodes (see \textit{e.g.}~\cite[Thm 9.7]{sch}). Here we reprove this fact by constructing a polynomial $F\in \R[\x,\y]$ with degree $d$ in each of $\x$ and $\y$ and the property 
that for every $y\in \R$, the $d$ roots of $F(\x,y)\in \R[\x]_{ \leq d}$ are the other $d$ nodes (among them possibly $\infty$)  of this unique quadrature rule
with $y$ as a node. 

Furthermore, we give symmetric determinantal representations of $F$, which again translates the problem of 
finding nodes of a quadrature rule into finding the generalized eigenvalues of a real symmetric matrix, 
\textit{i.e.} solving $\det(\x A-B)=0$ where $A,B$ are real symmetric matrices and $A$ is positive semidefinite 
(see \textit{e.g.}~\cite[Chapter 4]{eig}).
While the literature on Gaussian quadrature rules is vast, to our knowledge these 
formulas are new.

To construct $F$ and its determinantal representations, we consider three quadratic forms on $\R[\t]_{\leq d-1}$, 
namely those taking $p\in \R[\t]_{\leq d-1}$ to $\int p^2d\mu$,  $\int \t \cdot p^2d\mu$, and  $\int \t^2 \cdot p^2d\mu$, respectively. 
We call $M_{d-1}$, $M'_{d-1}$ and $M''_{d-1}$ the matrices representing these quadratic forms with respect to the
basis $1,\t,\hdots, \t^{d-1}$. That is, define the $d\times d$ matrices
\begin{equation}\label{eq:matrices}
M_{d-1}\ = \ (m_{i+j-2})_{1\leq i,j \leq d}, \ \ \ M'_{d-1}  \ = \ (m_{i+j-1})_{1\leq i,j \leq d}, \ \ \text{ and } \ \ M''_{d-1} \ = \  (m_{i+j})_{1\leq i,j \leq d}.
\end{equation}

\begin{thm}\label{thm:main}
Let $\mu$ be a measure on $\R$ that is non-degenerate in degree $d\geq1$.
There is a unique (up to scaling) polynomial $F\in \R[\x,\y]$ of degree $2d$ with the property that for $x,y\in \R$, $F(x,y)=0$ 
if and only if there is a Gaussian quadrature rule for $\mu$ of degree $2d$ with nodes 
$x=r_1,y=r_2,r_3, \hdots, r_{d+1} $ in $ \R\cup \{\infty\}$. 
This polynomial has the following two determinantal representations:
\begin{itemize}
\item[(a)] the determinant of the $d\times d$ matrix with bilinear entries in $\x,\y$,
\[ 
F  \ \ = \ \ \det( \x\y  M_{d-1} - (\x+\y)M'_{d-1} + M''_{d-1}),
\]
\item[(b)] the determinant of the $(2d+1)\times (2d+1)$ matrix with linear entries in $\x,\y$,
\[(\x-\y) \cdot F \ \ = \ \
c\cdot \det\begin{pmatrix}
       \frac{\det M_{d-1}}{\det M_d} (\x-\y) &e_d^T&e_d^T\\
   e_d&\x M_{d-1}-M_{d-1}'&0\\
   e_d&0&-\y M_{d-1}+M_{d-1}'
      \end{pmatrix}
\]
where $e_d=(0,\ldots,0,1)^T\in\mathbb{R}^d$ and $c=(-1)^d \det(M_d)/\det(M_{d-1})^2$.
\end{itemize}
Moreover, for all $x\in \R$, with $\det(x M_{d-1} -M_{d-1}')\neq 0$, all nodes are on the real line, i.e., the associated quadrature rule has no nodes at infinity.
\end{thm}

We prove the two parts of this theorem in Sections~\ref{sec:BilinearDetRep} and \ref{sec:LinearDetRep}.
The implications for finding quadrature rules as generalized eigenvalues are made explicit in 
Remarks~\ref{rem:GenEigBilinear} and \ref{rem:GenEigLinear}.
In Section~\ref{sec:Generalizations}, we discuss possible generalizations of this result.  

\subsection*{Acknowledgments}
This collaboration was initiated during the Oberwolfach workshop
``Real algebraic geometry with a view toward moment problems and optimization'' in March 2017.
Our thanks go to the organizers of the workshop and to the Mathematisches Forschungsinstitut Oberwolfach. 
We thank Petter Brändén, Christoph Hanselka, Rainer Sinn and Victor Vinnikov for interesting discussions 
during the workshop.
Grigoriy Blekherman was partially supported by  National Science Foundation (DMS-1352073).
Cordian Riener was partially supported by the Oberwolfach Fellowship program and  the  Troms\o~ Research Foundation (17\_matte\_CR).
Markus Schweighofer was partially supported by the German Research Foundation (SCHW 1723/1-1).
Cynthia Vinzant was partially supported by the National Science Foundation (DMS-1620014).

\section{Bilinear Determinantal Representation}\label{sec:BilinearDetRep}
In this section, we prove Theorem~\ref{thm:main}(a). The proof relies heavily on understanding 
a particular symmetric bilinear form $B_{x,y}$ on  $\R[\t]_{\leq d-1}$. 
For $x,y\in \R$ and $p,q\in \R[\t]_{\leq d-1}$, let
\[B_{x,y}(p,q) \ = \ \int (x-\t)(y-\t)  p q \ d\mu \  = \ xy\int p q \ d\mu-(x+y)\int  \t\cdot p q \ d\mu+\int  \t^2\cdot p q \ d\mu.\]
Note that the symmetric matrix $xy  M_{d-1} - (x+y)M'_{d-1} + M''_{d-1}$ represents $B_{x,y}$ with 
respect to the basis $1,\t,\t^2,\hdots, \t^{d-1}$. 
Define $F\in \R[\x,\y]$ to be the polynomial \[F\ \ = \ \ \det(\x\y  M_{d-1} - (\x+\y)M'_{d-1} + M''_{d-1}).\]
We will show that $F$ satisfies the requirements of Theorem~\ref{thm:main}. 

\begin{proof}[Proof of Theorem~\ref{thm:main}(a)] 
Suppose that there exists a quadrature rule for $\mu$ of degree $2d$  with nodes 
$r_1, r_2, \hdots, r_{d+1}$ in $\R\cup\{\infty\}$, meaning that there exist $w_1, \hdots, w_{d+1}\in \R_{>0}$ so that 
\begin{equation}\label{eq:nodes}
 \int f \ d\mu   \ \ = \ \ \sum_{i=1}^{d+1} w_i \ev_{r_i}(f) \ \ \text{ for all } \ \ f\in \R[\t]_{\leq 2d}.
 \end{equation}
Note that because $\mu$ is non-degenerate in degree $d$, 
the nodes $r_1, \hdots, r_{d+1}$ are necessarily distinct, so that after reindexing 
we may assume that $r_1, \hdots, r_d\in \R$. 

We claim that $F(r_1, r_2) = 0$. To see this, let $q$ be the unique (up to scaling) nonzero polynomial 
with $\deg(q)\leq d-1$ and $\ev_{r_i}(q)=0$ for each $i=3, \hdots, d+1$. If each $r_i\in \R$, 
then we can take $ q$ to be $(\t-r_3)\hdots (\t-r_{d+1})\in \R[\t]_{\leq d-1}$. 
If $r_{d+1}=\infty$, then we can take $ q= (\t-r_3)\hdots (\t-r_{d})$.
For any $p\in \R[\t]_{\leq d-1}$, it follows that
\[B_{r_1,r_2}(p,q) \ = \ \int (r_1-\t)(r_2-\t)  p q \ d\mu \  =\  0. \]  
The last equality follows from \eqref{eq:nodes} and the fact that $\deg(p)\leq d-1$. 
Therefore $q$ is an element of the right kernel of $B_{r_1, r_2}$. Since $B_{r_1, r_2}$ drops rank, the 
determinant $F(r_1, r_2)$ of the representing matrix equals zero.  

Conversely, suppose that for $x,y\in \R$, $F(x,y)=0$. 
Then the kernel of $B_{x,y}$ contains a polynomial $q\in \R[\t]_{\leq d-1}$. 
For all $p\in \R[\t]_{\leq d-1}$, 
\[  \int (\t-x)(\t-y)p q\ d\mu = 0.  \]
First, we argue that $x, y,$ and the roots of $q$ are real and pairwise distinct and 
that the degree of $q$ is $d-2$ or $d-1$. 
If not, there would exist a non-zero polynomial $p\in \R[\t]_{\leq d-1}$ 
for which $f=(\t-x)(\t-y)p q$ is nonnegative on $\R$.
Since $\int f \ d\mu =0$, this contradicts the assumption that $\mu$ is non-degenerate in degree $d$. 

Let $r_1=x$, $r_2=y$ and denote the roots of $q$ by $r_3, \hdots, r_{d+1}$, 
where we take $r_{d+1}=\infty$ if $\deg(q) = d-2$. Consider the conic hull of 
the $d+1$ points $\ev_{r_1}, \hdots, \ev_{r_{d+1}}$ in $\R[\t]_{\leq 2d}^*$.  
This is a $(d+1)$-dimensional simplicial convex cone in the $(2d+1)$-dimensional 
vectorspace $\R[\t]_{\leq 2d}^*$. 
Therefore this cone is defined by $d$ linear equalities and $d+1$ linear inequalities,
which we will now identify by inspection. 
For each $i=1, \hdots, d+1$, let $f_i$ be the unique (up to scaling) nonzero 
polynomial of degree $\leq d$ for which 
$\ev_{r_j}(f_i) = 0$ for each $j\neq i$. For example, $f_{d+1} = \prod_{i=1}^{d}(\t-r_i)$. 
Note that the polynomials $(\t-r_1)(\t-r_2)q\cdot \t^k$ for $0\le k\le d-1$ and $f_i^2$ for $1\le i\le d+1$ 
are linearly independent in $\R[\t]_{\leq 2d}$. It therefore follows that 
the conic hull of $\ev_{r_1}, \hdots, \ev_{r_{d+1}}$ is the set of $L\in \R[\t]_{\leq 2d}^*$ 
satisfying
\[
L((\t-r_1)(\t-r_2)pq) = 0 \ \text{ for all } \ p \in \R[\t]_{\leq d-1} 
\ \ \text{ and } \ \
L(f_i^2) \geq 0 \ \text{ for } \ i=1, \hdots d+1.
\]
The linear function $L_{\mu}\in \R[\t]_{\leq 2d}$ given by $L_{\mu}(f) = \int f\ d\mu$ 
belongs to this convex cone. Hence there are nonnegative weights $w_1, \hdots, w_{d+1}$ for which 
$L_{\mu}(f) = \sum_{i=1}^{d+1}w_i \ev_{r_i}(f)$.
Since $\mu$ is non-degenerate in degree $d$, each of these weights must be positive. 
\end{proof}

\begin{rem}\label{rem:PDdiag}
Let $M = M(\x, \y) = \x\y M_{d-1}-(\x+\y)M'_{d-1}+ M''_{d-1}$. We remark that 
for every $x\in \R$, the matrix $M(x,x)$ is positive definite. 
To see this, note that $M(x,x)$ represents the quadratic 
form on $\R[\t]_{\leq d-1}$ given by $p \mapsto \int p^2(\t-x)^2 d\mu$
with respect to the monomial basis. Since $\mu$ is non-degenerate in degree $d$, 
this is positive for any $0\neq p \in \R[\t]_{\leq d-1}$. 
\end{rem}

\begin{rem}\label{rem:GenEigBilinear}
For fixed $y\in \R$, this allows us to find the roots of $F(\x,y)\in \R[\x]$ as the generalized eigenvalues of a $d\times d$ 
real symmetric matrix.  If $y$ is larger than all of the roots of $\det(\y M_{d-1} - M'_{d-1})$, then 
$yM_{d-1} - M'_{d-1}$ is positive definite, meaning that 
$M(\x,y)$ has the form $\x A-B$ where $A$ is a positive definite matrix. 
We can find the roots in $\x$ as the following generalized eigenvalue problem: 
\[
\{x : F(x,y)=0\} \ \ = \ \ \{ x   \ : \  \det(x(yM_{d-1} - M'_{d-1}) - (yM'_{d-1} -M_{d-1}''))=0\}.
\]
If $yM_{d-1} - M'_{d-1}$ is not positive definite, this formula still holds, but may not be as numerically stable. 
We can instead make a change of variables $\x=\tilde{\x}+y$, which gives
\begin{align*}
M(\x,y) \ = \ M(\tilde{\x}+y,y) 
&\ = \  (\tilde{\x}+y)yM_{d-1}-(\tilde{\x}+2y)M'_{d-1}+ M''_{d-1} \\
& \ = \  \tilde{\x}(yM_{d-1} - M'_{d-1}) + M(y,y).
\end{align*}
In particular, this has the form $\tilde{\x}A+B$ where $B$ is positive definite. Suppose $\lambda$ is a root 
of $\det(\blambda B- A)$. Then $\tilde{\x}=-1/\lambda$ is a solution to $\det(\tilde{\x}A+B)=0$. Therefore 
\[\{x : F(x,y)=0\} \ \ = \ \ \{ y-1/\lambda   \ : \  \det(\lambda M(y,y) - yM_{d-1} + M'_{d-1})=0\}.\]
Note that this even works when $\lambda = 0$ if we take $y-1/\lambda$ to be $\infty$. 
\end{rem}

\begin{cor}\label{cor:Uniqueness}
Let $\mu$ be a measure on $\R$ that is non-degenerate in degree $d\geq1$.
For $y\in \R$ there is a unique quadrature rule for $\mu$ of degree $2d$ 
with $d+1$ nodes, one of which is $y$. 
\end{cor}

\begin{proof}
Let $y\in \R$. The polynomial $F(\x,y) = \det(M(\x,y)) \in \R[\x]$ has degree at most $d$.  Since the 
matrix pencil $\{M(x,y): x\in \R\}$ contains a positive definite matrix $M(y,y)$, the roots of 
$F(\x,y)=0$ must all be real. In particular the existence of one real root $x$ implies, by Theorem~\ref{thm:main},
the existence of a quadrature rule of degree $2d$ of $\mu$ whose $d+1$ nodes include $y$. 
Moreover the other $d$ nodes $x, r_3, \hdots, r_{d+1}$ are necessarily the $d$ roots of $F(\x,y)=0$. 
As in the proof of Theorem~\ref{thm:main}, the conic hull of $\ev_x, \ev_y, \ev_{r_3}, \hdots, \ev_{r_{d+1}}$
is a simplicial cone containing $L_{\mu}$, meaning that there is a unique representation of $L_{\mu}$
as a nonnegative combination of them. 
\end{proof}

\begin{ex}
[Normal Distribution, $d=3$] \label{ex:Normal}
Let $\mu$ be the normal Gaussian distribution on $\R$ with mean $0$ and variance $1$. 
Its moments are given by $m_{2i+1}=0$ and $m_{2i} = (2i-1)!!$ for $i\in \N_0$. 
For example, 
the first seven moments of this measure are $m_1=m_3=m_5=0$, $m_0=m_2=1$, $m_4=3$, and $m_6 = 15$. 
For $d=3$, 
the polynomial $F$ given by Theorem~\ref{thm:main} is the determinant of the $3\times3$ matrix polynomial
\[
M \ = \  \begin{pmatrix}
 \x \y+1 & -(\x+\y) & \x \y+3 \\
 -(\x+\y) & \x \y+3 & -3 (\x+\y) \\
 \x \y+3 & -3 (\x+\y) & 3 \x \y+15 \\
\end{pmatrix}.
\]
For fixed $y\in \R$, the polynomial $F(\x,y)\in \R[\x]$ has three real roots, $r_1, r_2, r_3 \in \R\cup\{\infty\}$, 
which, together with $y$, form the nodes of a quadrature rule for $\mu$ of degree $6$. 
The matrix $M(\x,y)$ has the form $\x A+B$ 
for some real symmetric matrices $A$ and $B$ (depending on $y$). 
The roots of $F(\x,y)$ are the roots of $\det(\x A + B)$. 

\begin{figure}
\begin{center} 
\includegraphics[height=3in]{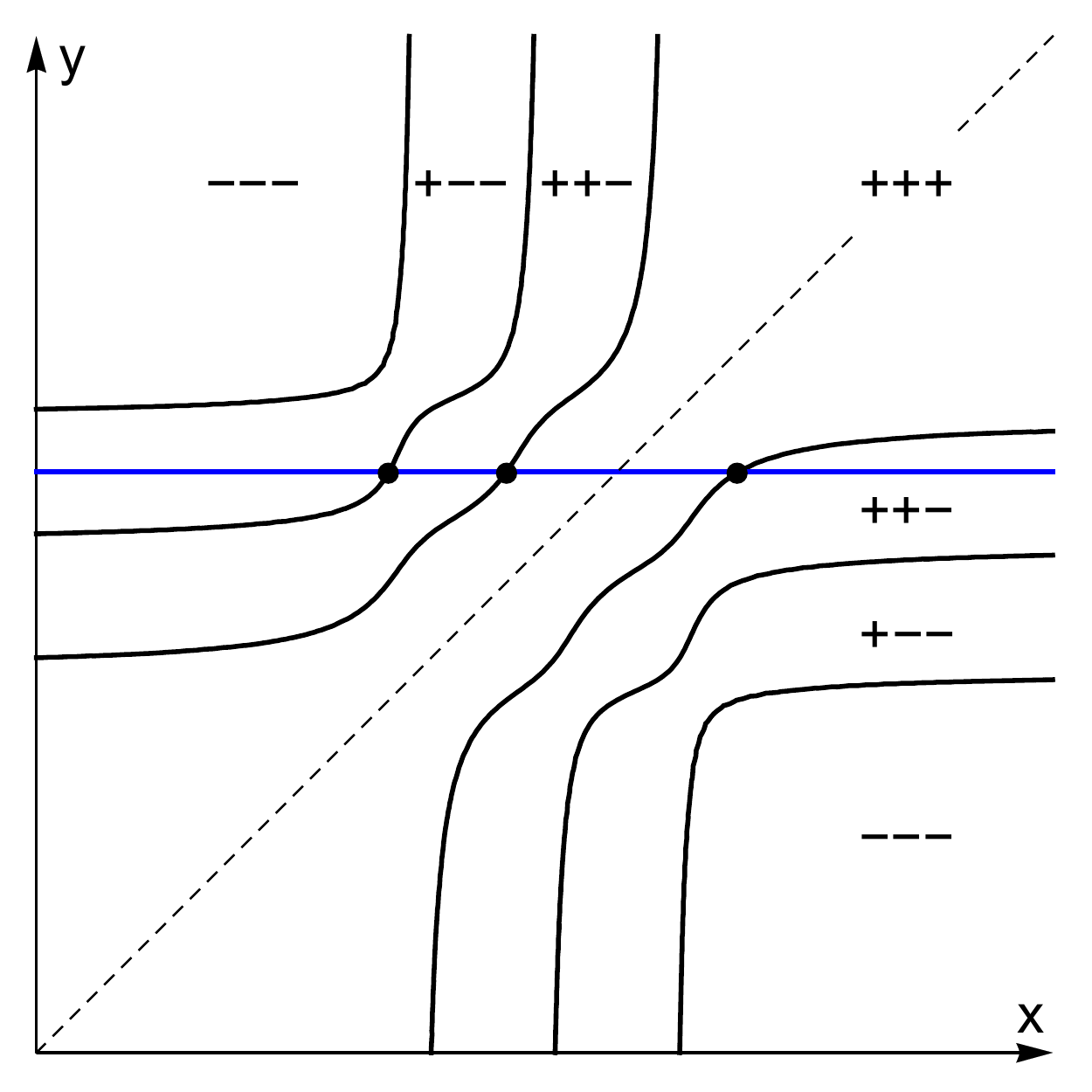}
\end{center}
\caption{\label{fig:Normal} 
The sextic curve given by $F=0$, the line $y=1$, and signs of eigenvalues of the $3\times 3$ matrix polynomial
$M$
 from Example~\ref{ex:Normal}.
}
\end{figure}

Moreover, by making a change of coordinates we can make $A$ positive definite. 
For example, for $y=1$, we set $\x =1/\blambda+1$ to get
 \[
\blambda \cdot M(1/\blambda+1,1) = 
\begin{pmatrix}
 2 \blambda+1 & -2 \blambda-1 & 4 \blambda+1 \\
 -2 \blambda-1 & 4 \blambda+1 & -6 \blambda-3 \\
 4 \blambda+1 & -6 \blambda-3 & 18 \blambda+3 \\
\end{pmatrix},
\]
which has the form $\blambda A-B$, where $A=M(1,1)$ is positive definite.  
Solving the generazlied eigenvalue problem $\det(\blambda A - B)=0$ 
gives $\blambda \approx -0.66, -0.32, 0.60$. The solutions to $F(\x,1)=0$ 
are then $\x = 1+1/\blambda \approx-2.15, -0.52, 2.67$. 
Thus there is a quadrature rule for $\mu$ of degree $6$ with 
nodes consisting of $1$ and these three roots. 
The curve given by $F(x,y)=0$ along with the line $\y=1$ are shown 
in Figure~\ref{fig:Normal}. 

For $y\in \{0,\pm \sqrt 3\}$, the polynomial $F(\x,y)\in \R[\x]$ becomes quadratic with two real roots 
$\{0,\pm \sqrt{3}\}\backslash\{y\}$. From this we conclude that there is a quadrature rule of degree $6$ 
for $\mu$ with the four nodes $0,\pm \sqrt 3, \infty$. 

For $y>\sqrt{3}$, the matrix coefficient of $\x$ in $M(\x,y)$ is positive definite. 
In this case, $M(\x,y)$ already has the form $\x A-B$ where $A$ is positive definite, so 
no change of coordinates is required to translate this into a generalized eigenvalue problem.  
\hfill $\diamond$
\end{ex}

\section{Linear Determinantal Representation}\label{sec:LinearDetRep}

In this section we prove Theorem~\ref{thm:main}(b). As in Section~\ref{sec:BilinearDetRep}, we construct a bilinear form depending 
on $x,y\in \R$ and construct a non-zero kernel of it for 
 those pairs $x,y$ that can be extended to $d+1$ nodes of a quadrature rule
for $\mu$ of degree $2d$.

For $x,y\in \R$, we define a bilinear form $B_{x,y}$ on $\R \oplus \R[\t]_{\leq d-1} \oplus \R[\t]_{\leq d-1} \cong \R^{2d+1}$. 
Given $p = (p_0, p_1, p_2)$ and $q = (q_0,q_1,q_2)$ in $\R \oplus \R[\t]_{\leq d-1} \oplus \R[\t]_{\leq d-1}$, 
define 
\begin{align*}
B_{x,y}(p,q) = & \int \! p_1q_1 (x-\t) + p_2q_2 (\t-y)  d\mu \\
&+ \ev^{d-1}_{\infty}(q_0(p_1+p_2)+p_0(q_1+q_2)) +\frac{p_0q_0\det(M_{d-1})}{\det(M_d)}(x-y).
\end{align*}
Choosing the basis $1,\t,\hdots, \t^{d-1}$ for both copies of $\R[\t]_{\leq d-1}$ represents this 
symmetric bilinear form as the $(2d+1)\times (2d+1)$ matrix given in Theorem~\ref{thm:main}(b). 
That is, if $\vec f$ denotes the coefficients of $f\in \R[\t]_{\leq d -1}$ so that $f = \vec{f} \cdot (1,\t,\hdots, \t^{d-1})$,  
then 
\[
B_{x,y}(p,q) \ = \
\begin{pmatrix}
p_0 \\ \vec{p_1} \\ \vec{p_2}
\end{pmatrix}^T
\begin{pmatrix}
       \frac{\det M_{d-1}}{\det M_d} (x-y) &e_d^T&e_d^T\\
   e_d&xM_{d-1}-M_{d-1}'&0\\
   e_d&0&-yM_{d-1}+M_{d-1}'
      \end{pmatrix}
\begin{pmatrix}
q_0 \\ \vec{q_1} \\ \vec{q_2}
\end{pmatrix},
\]
where $M_{d-1}$ and $M'_{d-1}$ are the $d\times d$ matrices defined in \eqref{eq:matrices}.

In order to construct a non-zero element in the kernel of $B_{x,y}$, we need to build up some basic facts. The first concerns
quadrature rules for $\mu$ whose nodes include $\infty$. 
As this polynomial will be used heavily in the text below, denote
\[F_{\infty} = \det(\x M_{d-1} -M_{d-1}').\]

\begin{lem}\label{lem:InftyNodes}
The polynomial $F_{\infty}$ has $d$ real roots $s_1, \hdots, s_d \in \R$
and there exist weights $w_1, \hdots, w_d\in \R_{> 0}$ for which 
\begin{equation}\label{eq:InftyQuad}
\int f \ d\mu  \ \ = \ \  w_{\infty} \ev_{\infty}^{2d}(f) + \sum_{i=1}^d w_i \ev_{s_i}(f) \ \ \text{ for all } \ \ f\in \R[\t]_{\leq 2d},
\end{equation}
where $w_{\infty} = \det(M_d)/\det(M_{d-1})$. 
\end{lem}
\begin{proof} 
Since $M_{d-1}$ is positive definite, $F_{\infty}$ has $d$ real roots $s_1, \hdots, s_d\in \R$, up to multiplicity. 
For any root $s$, 
the matrix $sM_{d-1} - M'_{d-1}$ has rank $<d$, meaning that the polynomial 
\[F(\x,s) = \det(\x(sM_{d-1} - M'_{d-1}) - (sM'_{d-1} -M_{d-1}''))\]
has degree $\leq d-1$ in $\x$.  For any $r\in \R$ with $F(r,s)=0$, there is a quadrature
rule for $\mu$ of degree $2d$ with $d+1$ nodes containing $s$, $r$, and $\infty$. 
Then the unique quadrature rule of degree $2d$ with $d+1$ nodes containing $r$ 
also contains the node $\infty$. This implies that $F(\x,r)$ has degree $\leq d-1$ and $F_\infty(r)=0$, 
meaning that $r\in \{s_1, \hdots, s_d\}$. Therefore there is a quadrature rule of degree $2d$ for $\mu$
with nodes $s_1, \hdots, s_d, \infty$. In particular, there exist $w_1, \hdots, w_d, w_\infty\in \R_{>0}$ for which 
\eqref{eq:InftyQuad} holds. 

Then $w_\infty$ is the largest $\lambda$ for which the quadratic 
form $p \mapsto \int p^2 d\mu - \lambda \ev_{\infty}(p^2)$ is nonnegative on $\R[\t]_{\leq d}$.  
This is the largest $\lambda$ for which the matrix $M_d - \lambda e_{d+1}e_{d+1}^T$ is positive semidefinite. 
We find this by solving the equation $\det(M_d - \blambda e_{d+1}e_{d+1}^T) = 0$, which gives 
$\blambda = w_{\infty} = \det(M_d)/ \det(M_{d-1})$.
\end{proof}

\begin{lem}\label{lem:kernel}
Let $w_{\infty}, s_1, \hdots, s_d\in \R$ be as given by Lemma~\ref{lem:InftyNodes}.
If $y\in \R$ with $F_{\infty}(y)\neq 0$, 
then there is a quadrature rule for $\mu$ of degree $2d$ with nodes 
$y, r_1,\hdots, r_d\in  \R$. 
Let
\[
q_y  \ \ = \ \  \prod_{i=1}^{d}(\t-s_i) - \prod_{j=1}^d(\t-r_j)\ \ \in \ \ \R[\t]_{\leq d-1}.
\] 
Then for all $p\in \R[\t]_{\leq d-1}$, 
\[
\int p \cdot q_y \cdot (\t-y) \ d\mu \ \ = \ \ w_\infty \ev^{d-1}_{\infty}(p).
\]
\end{lem}
\begin{proof}
By Corollary~\ref{cor:Uniqueness}, there is a quadrature rule for $\mu$ of degree $2d$ with nodes 
$y, r_1,\hdots, r_d$ in $\R\cup \{\infty\}$. 
Since $F_{\infty}(y)\neq 0$ the univariate polynomial $F(\x,y)\in\R[\x]$ has degree $d$ and by the uniqueness of the quadrature rule, $r_1,\hdots, r_d$ are its roots. In particular, $r_j\in \R$.
Then for $p\in \R[\t]_{\leq d-1}$, we have that 
\begin{align*}
\int p \cdot q_y \cdot (\t-y) \ d\mu
\ &= \  \int p \cdot  \prod_{i=1}^{d}(\t-s_i) \cdot (\t-y) \ d\mu
- \int p \cdot  \prod_{j=1}^d(\t-r_j)\cdot (\t-y) \ d\mu\\
&= \  \int p \cdot  \prod_{i=1}^{d}(\t-s_i) \cdot (\t-y) \ d\mu\\
&= \ w_{\infty}\ev^{2d}_{\infty}\left(p \cdot  \prod_{i=1}^{d}(\t-s_i) \cdot (\t-y)\right)\\
&= \ w_{\infty}\ev^{d-1}_{\infty}(p).
 \end{align*}
The first equality comes from the fact 
that there is a quadrature rule for $\mu$ 
of degree $2d$ with nodes $y,r_1, \hdots, r_d$.
The second follows from the equality in Lemma~\ref{lem:InftyNodes}.
 \end{proof}

We now make a special choice of $q$ and use Lemma~\ref{lem:kernel}
to greatly simplify $B_{x,y}(p,q)$.

\begin{lem}\label{lem:q}
Suppose $x,y\in \R$ satisfy $F_{\infty}(x)\neq0$ and $F_{\infty}(y)\neq0$. Take $w_\infty\in \R$ and 
$q_x, q_y\in \R[\t]_{\leq d-1}$ as given by Lemmas~\ref{lem:InftyNodes} and \ref{lem:kernel} and fix
 $q = (w_{\infty}, q_x, -q_y)$.  Then for any 
$p = (p_0, p_1, p_2)\in \R\oplus \R[\t]_{\leq d-1}\oplus \R[\t]_{\leq d-1}$, 
\[
B_{x,y}(p,q)  \ \ = \ \ p_0\cdot \ev_{\infty}(q_x-q_y) +p_0 \cdot (x-y).
\]
\end{lem}
\begin{proof}
Let  $q = (w_{\infty}, q_x, -q_y)$ and suppose $p = (p_0,p_1, p_2) \in \R \oplus \R[\t]_{\leq d-1} \oplus \R[\t]_{\leq d-1}$. 
By Lemma~\ref{lem:InftyNodes}, $w_{\infty} $ equals $\det(M_d)/\det(M_{d-1})$. 
Then, by definition, 
\[
B_{x,y}(p,q) =\int \! p_1q_x (x-\t) - p_2q_y (\t-y)  d\mu + \ev_{\infty}(w_{\infty}(p_1+p_2)+p_0(q_x-q_y)) +p_0(x-y).\\
\]
Using the properties of $q_x, q_y$ given in Lemma~\ref{lem:kernel}, this simplifies to
\begin{align*}
B_{x,y}(p,q)& = -w_{\infty}\ev_{\infty}(p_1) -w_{\infty} \ev_{\infty}(p_2)  +\ev_{\infty}(w_{\infty}(p_1+p_2)+p_0(q_x-q_y)) +p_0(x-y)\\
& =  \ev_{\infty}(p_0(q_x-q_y)) +p_0(x-y),
\end{align*}
as claimed.
\end{proof}

\begin{lem}\label{lem:zard}
 Let $H\in\R[\x,\y]$ be a non-zero polynomial with the property that for every $y\in\R$ the polynomial 
 $H(\x,y)\in\R[\x]$ has distinct, real zeros. Then every polynomial vanishing on the real variety of $H$ must be a polynomial multiple of $H$.
\end{lem}

\begin{proof}
 Let $H=H_1\cdots H_r$ where each $H_i$ is irreducible in $\C[\x,\y]$. 
By our assumption on distinct zeros, the $H_i$ are pairwise coprime. 
Moreover, each factor $H_i$ belongs to $\R[\x,\y]$. If not, then since $H\in \R[\x,\y]$, 
both $H_i$ and its complex conjugate $\overline{H_i}$ appear as factors of $H$
and have the same real roots along the line $\y = y$ for $y\in \R$, contradicting the distinctness of these roots.
Thus it suffices to show that a polynomial vanishing on the real variety of $H$ must be a polynomial multiple of each $H_i$. Since each $H_i$ satisfies our assumption as well, we can assume without loss of generality that $H$ is irreducible itself.
 
 The real variety of $H$ contains infinitely many points so its Zariski closure in $\C^2$ 
 is at least one dimensional. By irreducibility of $H$, its real variety Zariski-dense in its complex variety. 
 Now the claim follows from Hilbert's Nullstellensatz.
\end{proof}

Now we are ready to prove Theorem~\ref{thm:main}(b).

\begin{proof}[Proof of Theorem~\ref{thm:main}(b)] 
Let $F$ be the polynomial given by Theorem~\ref{thm:main}(a).
Then $F$ has degree $2d$ in $\x$ and $\y$, with top degree part $\det(M_{d-1})\x^d\y^d$, 
which is non-zero by the non-degenerateness of $\mu$. 
Then $(\x-\y)F$ has degree $2d+1$ with top degree part equal to $\det(M_{d-1})(\x-\y)\x^d\y^d$.
Also, for every $y\in \R$, the polynomial  $(\x-y)F(\x,y)\in \R[\x]$ has distinct, real roots. 
Lemma \ref{lem:zard} then implies that any polynomial $G\in \R[\x,\y]$ vanishing on the real variety of $(\x-\y)F$ 
must be a polynomial multiple of it. 

We further claim that 
the points $(x,y)$ in $V_\R((\x-\y)F)$ with $F_{\infty}(x)\neq0$ and 
 $F_{\infty}(y)\neq0$ are Zariski-dense in $V_\R((\x-\y)F)$.  
 That is, any polynomial $G\in \R[\x,\y]$ vanishing on 
 $V_\R((\x-\y)F)\backslash V_\R(F_{\infty}(\x)F_{\infty}(\y))$
 also vanishes on $V_\R((\x-\y)F)$ and is therefore a multiple of $(\x-\y)F$.
 It suffices to show that $(\x-\y)F$ has no factors in common with $F_{\infty}(\x)F_{\infty}(\y)$. 
 The factors of $F_{\infty}(\x)F_{\infty}(\y)$ are given by Lemma~\ref{lem:InftyNodes}.
 Suppose for the sake of contradiction that $(\x-s_i)$ divides $(\x-\y)F$ for some $i=1,\hdots, d$. 
It must be that $(\x-s_i)$ divides $F$.  This implies that 
 $F(s_i, s_i)$ equals zero, which contradicts the observation in Remark~\ref{rem:PDdiag} that the matrix $M(s_i,s_i)$ 
 is positive definite and $F(s_i,s_i)= \det(M(s_i,s_i))>0$.

Let $G$ denote the determinant of the $(2d+1)\times (2d+1)$ matrix representing the bilinear form $B_{\x,\y}$.
We will show that $(\x-\y)F = c\cdot G$ where $c = (-1)^d \det(M_d)/\det(M_{d-1})^2$.
Note that $c \cdot G$ is also a polynomial of degree $2d+1$. Inspection shows that its top degree part 
is  $\det(M_{d-1})(\x-\y)\x^d\y^d$.  Thus by the above argument, it suffices to show that $G(x,y)=0$ for all 
$(x,y)\in V_\R((\x-\y)F)$ with $F_{\infty}(x)F_{\infty}(y)\neq 0$. 

Now take $x, y\in \R$ with $(x-y)F(x,y)=0$ and $F_{\infty}(x)F_{\infty}(y)\neq 0$.
Let $q_x, q_y\in \R[\t]_{\leq d-1}$ be the polynomials given by Lemma~\ref{lem:kernel} and 
let $q = (w_{\infty}, q_x, -q_y)$. We claim that $q$ is in the kernel of $B_{x,y}$. To see this, let 
$p = (p_0, p_1, p_2)\in \R\oplus \R[\t]_{\leq d-1} \oplus \R[\t]_{\leq d-1}$. Then 
\[B_{x,y}(p,q)  \ \ = \ \ p_0\cdot \ev_{\infty}(q_x-q_y) +p_0 \cdot (x-y),\]
by Lemma~\ref{lem:q}. If $x=y$, this is clearly zero. Otherwise $F(x,y)=0$ and 
by Theorem~\ref{thm:main}(a) there exists $r_3, \hdots, r_{d+1}\in \R \cup \{\infty\}$
and a quadrature rule of degree $2d$ for $\mu$ with nodes $x,y,r_3, \hdots,r_{d+1}$. 
Since $F_{\infty}(x)\neq 0$, each $r_i\in \R$. 
Then 
\[
q_x  \  = \ \prod_{i=1}^{d}(\t-s_i) - (\t-y)\prod_{j=3}^{d+1}(\t-r_j)
\ \ \text{ and } \ \ 
q_y  \  =  \  \prod_{i=1}^{d}(\t-s_i) - (\t-x)\prod_{j=3}^{d+1}(\t-r_j).
\]
Expanding and looking at the coefficient of $\t^{d-1}$ reveals that 
\[
 \ev_{\infty}(q_x) = -\sum_{i=1}^d s_i + y+ \sum_{j=3}^{d+1} r_j 
 \ \ \text{ and } \ \ 
 \ev_{\infty}(q_y) = -\sum_{i=1}^d s_i + x+ \sum_{j=3}^{d+1} r_j . 
\]
In particular,  $\ev_{\infty}(q_x -q_y) = y-x$, giving $B_{x,y}(p,q)=0$.
Since the bilinear form $B_{x,y}$ has a non-zero kernel, the determinant, $G(x,y)$, 
of its representing matrix is zero. 
\end{proof}

\begin{remark}\label{rem:GenEigLinear}
Again this translates the problem of finding the roots of $F(\x,y)$ for fixed $y\in \R$ 
into a generalized eigenvalue problem. Consider the $(2d+1)\times (2d+1)$ symmetric matrix polynomial
\[
\mathcal{M}  \ = \  \mathcal{M}(\x, \y)  \ = \  \begin{pmatrix}
       \frac{\det M_{d-1}}{\det M_d} (\x-\y) &e_d^T&e_d^T\\
   e_d&\x M_{d-1}-M_{d-1}'&0\\
   e_d&0&-\y M_{d-1}+M_{d-1}'
      \end{pmatrix}.
\]
Since $M_{d-1}$ is positive definite and $\det(M_{d-1})/\det(M_d)$ is positive, the coefficient of $\x$ 
in $\mathcal{M}$ 
is positive semidefinite of rank $d+1$. In particular, for fixed $y\in \R$, we can solve $F(\x,y)=0$ by 
solving $0=\det(\mathcal{M}(\x,y)) = \det(\x A-B)$, where $A$ is positive semidefinite.  
\end{remark}

\begin{ex}[Normal Distribution, $d=3$] \label{ex:Normal2}
Consider again $d=3$ and the normal distribution given in Example~\ref{ex:Normal}. 
We calculate that $\det(M_3) = 12$ and $\det(M_2) = 2$. 
The degree $6$ polynomial $F$ given by Theorem~\ref{thm:main} 
satisfies $(\x-\y)F = -3\det(\mathcal{M})$ where 
\[
{\small
\mathcal{M}  \ = \  \begin{pmatrix}
\frac{\x-\y}{6} 	& 0 & 0 & 1& 0 & 0 & 1\\
 0		& \x & -1 & \x & 0 & 0 & 0 \\
 0		& -1 & \x & -3 & 0 & 0 & 0 \\
 1		& \x & -3 & 3 \x & 0 & 0 & 0 \\
 0		& 0 & 0 & 0 & -\y & 1 & -\y \\
 0		& 0 & 0 & 0 & 1 & -\y & 3 \\
 1		& 0 & 0 & 0 & -\y & 3 & -3 \y \\
\end{pmatrix}.
}
\]
Then $\mathcal{M}(\x,1)$ has the form $\x A - B$ where $A$ is a positive semidefinite matrix of rank four. 
The determinant $\det(\mathcal{M}(\x,1))$ has four roots, $\x\approx -2.15, -0.52, 1, 2.67$, which are the 
nodes of a quadrature rule for $\mu$ of degree $6$.  The curve $\det(\mathcal{M})=0$ along with the line $\y=1$ 
are shown in Figure~\ref{fig:Normal2}.
\hfill $\diamond$
\end{ex}

\begin{figure}
\begin{center} 
\includegraphics[height=3in]{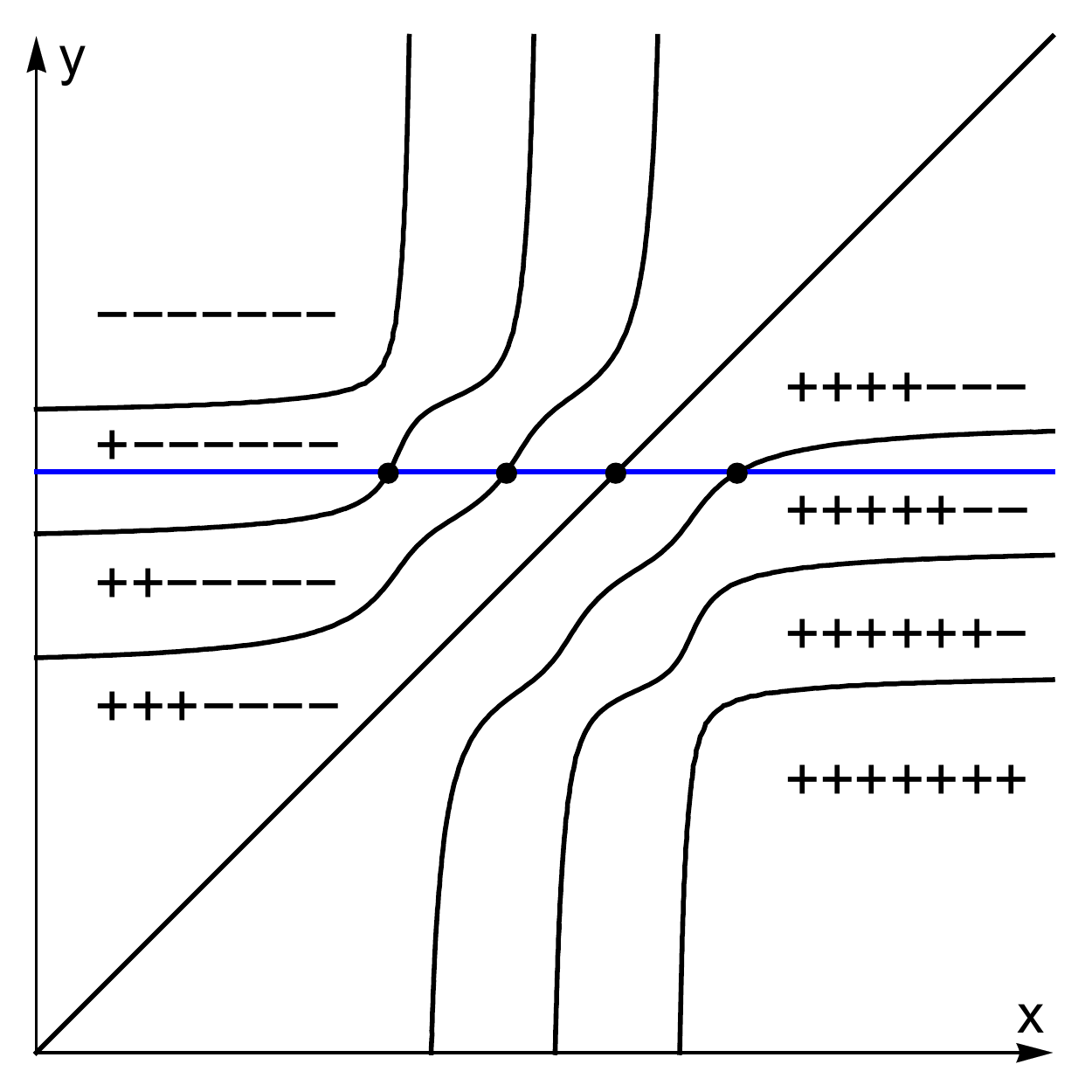}
\end{center}
\caption{\label{fig:Normal2} 
The curve given by $\det(\mathcal{M})=0$ and signs of eigenvalues of the $7\times 7$ matrix 
$\mathcal{M}$
 from Example~\ref{ex:Normal2}.
}
\end{figure}

\begin{rem}
Note that the matrix $\mathcal{M}(\x,-\y)$ has the form $\x A + \y B + C$ where the matrices $A,B$ are positive semidefinite.  For any $a_1, a_2 \in \R_{>0}$ and $b_1,b_2\in \R$, the matrix coefficient of $\t$ in 
$\mathcal{M}(a_1\t + b_1,-a_2\t + b_2)$ is positive definite, implying that 
the polynomial $F(a_1\t + b_1,-a_2\t + b_2)\in \R[\t]$ is real-rooted.  
One can see this in Figures \ref{fig:Normal} and \ref{fig:Normal2}, as 
 any line with negative slope intersects the curve $V(F)$ in six real points. 
It also shows the polynomial $F(\x,-\y)$ is \textit{real stable} \cite[Prop.~2.4]{wag}.
The Helton-Vinnikov theorem \cite[Thm. 2.2]{hv} then implies that
not only $(\x-\y)\cdot F$ has a $(2d+1)\times(2d+1)$ linear determinantal representation like in the Theorem \ref{thm:main}(b)
but also $F$ itself has a determinantal representation
$F = \det(\x A + \y B + C)$, where $A,B,C$ are $2d\times2d$ real symmetric matrices 
and $A$ and $-B$ are positive semidefinite. 
However it is unclear if there exists such a representation for which the entries 
of $A,B,C$ are easily calculated from the moments $m_{k}$ of $\mu$, as with $\mathcal{M}$.
\end{rem}

\section{Univariate quadrature rules with more nodes}\label{sec:Generalizations}

It is natural to try to generalize the above discussion to situations where more nodes of a quadrature rule are prescribed. Finding a quadrature rule means specifying an even number of real parameters since each node comes with a weight. We will now consider the following \textit{minimal  problem}:

\begin{problem} \label{prob:UnivariateGeneralization}
For integers $n,\ell \geq 1$,
 we are given $n+2\ell+1$ moments $m_0,\dots, m_{n+2\ell}$ of a (positive Borel) measure $\mu$ on the real line 
 that is non-degenerate in degree $n+2\ell$  
 and $n-1$ real numbers $x_1,\dots, x_{n-1}$. Does there exist a quadrature rule for $\mu$ of degree $n+2\ell$
 with $n+\ell$ nodes including  $x_1,\dots , x_{n-1}$?
\end{problem}

Specifying the quadrature rule requires $2(\ell+1)+n-1=n+2\ell+1$ parameters, as we have two parameters for each of $\ell+1$ unspecified nodes and one parameter for the weight of each of the $n-1$ specified nodes. Therefore the number of parameters that we have to choose is exactly equal to the number of moments that we need to match.

For $n=1$, the problem is solved by the well-known Gaussian quadrature. The case $n=2$ is the main focus of this paper, and it is classically known that such a measure exists. 
Unfortunately, for $n=3$ this can fail even if the measure is non-degenerate in every degree. To see this, we need some preparation.

For each integer $0 \leq k \leq n$, consider the quadratic form 
$p \mapsto \int \t^k  \cdot p^2 \ d  \mu$ on $\R[\t]_{\leq \ell}$, which, with respect to the basis 
$1,\t,\hdots, \t^{\ell}$, is represented by the 
$(\ell+1)\times (\ell+1)$ matrix
\[
M^{(k)}_{\ell} \ = \ (m_{i+j+k-2})_{1\leq i,j \leq \ell+1}.
\]
Notice that the highest moment of $\mu$ needed to specify these matrices is $m_{n+2\ell}$, 
achieved by $i = j = \ell+1$ and $k = n$. 

For $k\in\N_0$, we denote by $e_k(\x_1,\dots,\x_n) \in \R[\x_1, \hdots, \x_n]$ the
$k$-th elementary symmetric polynomial in $\x_1, \hdots, \x_n$, 
for which the following polynomial identity in $\R[\t]$
holds:
\[\prod_{i=1}^n(\t-\x_i) \ \ = \ \ \sum_{k=0}^n(-1)^ke_k(\x_1,\dots,\x_n)\t^{n-k}.\]

\begin{prop}\label{multiprop}
Let $\mu$ be a measure on $\R$ that is non-degenerate in degree $n+2\ell \geq 1$ for 
some  $n,\ell \in\N_0$ and let $x_1, \hdots, x_{n} \in \R$ be distinct. 
If there is a quadrature rule for $\mu$ of degree $n+2\ell$ 
with nodes $r_1 = x_1, \hdots, r_n = x_n, r_{n+1},\dots,r_{n+\ell}\in\R\cup\{\infty\}$ then
\[\det\left(\sum_{k=0}^n (-1)^k e_k(x_1,\dots,x_n) M^{(n-k)}_{\ell}\right)=0.\]
\end{prop}

\begin{proof}
For $X = \{x_1, \hdots, x_{n}\} \subset \R$, 
consider the bilinear form $B_{X}$ on $\R[\t]_{\leq \ell}$ given by 
\[
B_X(p,q) \ \ = \ \  \int \ \prod_{k=1}^{n}(\t-x_k)  \cdot p\cdot q \ d\mu
\ \ = \ \ \sum_{k=0}^n (-1)^k e_k(x_1,\dots,x_n) \cdot \int \t^{n-k} \cdot p\cdot q \ d\mu.
\]
With respect to the basis $1,\t, \hdots, \t^{\ell}$, this is represented by the matrix
\begin{equation}\label{eq:MultiLinearMatrix}
\sum_{k=0}^n (-1)^k e_k(x_1,\dots,x_n) M^{(n-k)}_{\ell}.
\end{equation}
Suppose that $r_1 = x_1, \hdots, r_n = x_n, r_{n+1},\dots,r_{n+\ell}\in\R\cup\{\infty\}$ are the nodes of 
a quadrature rule of degree $n+2\ell$ for $\mu$.
As in the proof of Theorem~\ref{thm:main}(a), let $q$ be the unique (up to scaling) 
nonzero polynomial in $\R[\t]_{\leq \ell}$ 
for which $\ev_{r_i}(q)=0$ for each $i = n+1, \hdots, n+\ell$. Then for every $p\in \R[\t]_{\leq \ell}$, 
\[
B_X(p,q) \ \ = \ \  \int \ \prod_{k=1}^{n}(\t-x_k)  \cdot p\cdot q \ d\mu \ \ = \ \ 0.
\]
Therefore the bilinear form $B_X$ has a nonzero kernel and the determinant of 
its representing matrix \eqref{eq:MultiLinearMatrix} is zero. 
\end{proof}

Unfortunately, the converse of Proposition~\ref{multiprop} does not hold for $n>2$. 

\begin{ex}
Consider the measure given by the exponential distribution on $\R$,  
whose $k$-th moment is $m_k = k!$ for every $k\in \N_0$. 
Let $n = 3$, $\ell=3$ so that $n+2\ell=9$. 
Then in Problem~\ref{prob:UnivariateGeneralization},
we want to build a quadrature rule for $\mu$ of degree $9$ 
with at most $ n+\ell=6$ nodes including $n-1=2$ specified nodes $x_1, x_2 \in \R$. 
However for $x_1=\frac13$ and $x_2=11$, 
the determinant of the $4\times 4$ matrix in Proposition \ref{multiprop} equals (up to a positive constant)
\[137503\x_3^4 -1695024\x_3^3+11282760\x_3^2-41197920\x_3+46998216 \  \in \ \R[\x_3],\]
which has complex roots $\x_3 \approx 1.87, 5.20, 2.63 \pm i \ 5.31$, only two of which are real.  
Since $\mu$ is non-degenerate in every degree, 
any quadrature rule for $\mu$ in degree $9$ has at $\geq 5$ nodes. 
However, by Proposition \ref{multiprop} there are only two
possibilities for the remaining  $\geq 3$ nodes. 
Therefore there is no quadrature rule for $\mu$ 
of degree $9$ with $\leq 6$ nodes including $x_1=\frac13$ and $x_2=11$.
\end{ex}


\begin{thebibliography}{BDDRV}


\bibitem[BDDRV]{eig} Z. Bai , J. Demmel , J. Dongarra , A. Ruhe, H. van der Vorst (Eds.):
Templates for the solution of algebraic eigenvalue problems: a practical guide,
Society for Industrial and Applied Mathematics, 2000

\bibitem[HV]{hv} J.W. Helton, V. Vinnikov:
Linear matrix inequality representation of sets,
Comm. Pure Appl. Math. 60 (2007), no. 5, 654--674

\bibitem[La1]{la1} M. Laurent:
Sums of squares, moment matrices and optimization over polynomials,
Emerging applications of algebraic geometry, 157--270, 
IMA Vol. Math. Appl., 149, Springer, New York, 2009

\bibitem[La2]{la2} M. Laurent:
Sums of squares, moment matrices and optimization over polynomials,
unpublished update of \cite{la1}, February 6, 2010
[\url{http://homepages.cwi.nl/~monique/files/moment-ima-update-new.pdf}]

\bibitem[Sch]{sch}
K. Schmüdgen: The moment problem, Graduate Texts in Mathematics, 277, Springer, Cham, 2017

\bibitem[Sze]{sze}
G. Szegő: Orthogonal polynomials,
Fourth edition, AMS Colloquium Publications, Vol. XXIII. American Mathematical Society, Providence, R.I., 1975

\bibitem[Wag]{wag} D.G. Wagner:
Multivariate stable polynomials: theory and applications,
Bull. Amer. Math. Soc. (N.S.) 48 (2011), no. 1, 53--84

\end{thebibliography}
\end{document}